\documentclass{amsart}
\usepackage{amsthm}
\usepackage[utf8]{inputenc}
\usepackage{appendix}
\usepackage{cite}
\usepackage{amssymb,amsmath,amsthm}
\usepackage[dvipsnames]{xcolor}
\newcommand{\amsprimary}[1]{{\footnotesize\noindent AMS 2010 \textit{Mathematics subject
classification:} Primary #1\vspace{1pc}}}
\newcommand{\keywordsnames}[1]{{\footnotesize\noindent\textit{Key words:} #1\vspace{1pc}}}

\newtheorem{theorem}{Theorem}
\newtheorem{teo}{Theorem}
\newtheorem{prop}[teo]{Proposition}

\newtheorem{lemma}[teo]{Lemma}

\theoremstyle{definition}

\theoremstyle{remark}

\title{The Yamabe equation in small convex domains in $\mathbb{R}^3$ and
small balls in $\mathbb{R}^n$}
\author{Jean C. Cortissoz \and Jonat\'an Torres-Orozco}
\date{}

\begin{document}

\maketitle

\begin{abstract}
We show an iterative method to solve a Dirichlet problem for a 
Yamabe-type equation in small convex domains in $\mathbb{R}^3$
and small balls in $\mathbb{R}^n$. 
   
\end{abstract}
 
{\keywordsnames {Yamabe equation, Yamabe problem, nonlinear elliptic equations, Dirichlet boundary conditions.}}

{\amsprimary {35J60, 53B20}}

\section{Introduction}

Let $(M, g)$ be a Riemannian manifold of dimension $n\geq 3$, with scalar curvature $R_g$. The scalar curvature of any conformal metric to $g$, say $\tilde{g}=e^{2f}\cdot g$, is given by
\[
R_{\tilde{g}}=e^{-2f}\left(R_g+2(n-1)\triangle_g f-(n-1)(n-2)|\nabla f|^2 \right)
\]
in terms of the Laplacian operator $\triangle_g$  and the Levi-Civita connection $\nabla$, with respect to $g$, for some smooth positive function $f$. Therefore the metric $\tilde{g}$ has constant scalar curvature $\lambda$ if and only if $f$ satisfies the {\em Yamabe equation}:
\begin{equation}
\label{eqn: Yamabe_grad}
2(n-1)\triangle_g f-(n-1)(n-2)|\nabla f|^2+R_g=\lambda\cdot e^{2f}.
\end{equation}
After an appropiate substitution, the gradient term can be removed. In fact, for $n\geq3$, if we write $u=\frac{2}{n-2}\log(f)$, the Yamabe equation becomes:
\begin{equation}
\label{eqn: Yamabe_nongrad}
-\frac{4(n-1)}{n-2}\triangle_g u+R_g u= \lambda\cdot u^{\frac{n+2}{n-2}}.
\end{equation}
Therefore, solving (\ref{eqn: Yamabe_nongrad}) is equivalent to solving  (\ref{eqn: Yamabe_grad}). We are interested in smooth positive solutions.
\medskip

The latter equation, treated as the standard form of the Yamabe equation, has been extensively studied. The fundamental result is the existence of at least one positive solution on closed Riemannian manifolds. This follows from the works of H. Yamabe \cite{Yamabe}, N. Trudinger \cite{Trudinger}, T. Aubin \cite{Aubin} and R. Schoen \cite{Schoen}. Many other settings have been explored. 
\newline

In this paper we consider the following
Yamabe-type Dirichlet problem on a small convex domain $\Omega$:
\begin{equation}
\label{eq:Yamabe_nonestandard}
    \left\{
    \begin{array}{l}
    \Delta f = -\dfrac{1}{2\left(n-1\right)}\left[e^{2f}R\left(x\right)+\left(n-1\right)
    \left(n-2\right)\left|\nabla f\right|^2\right]+S\left(x\right),
    \quad \mbox{in} \quad \Omega\\
    f=0,\quad \mbox{on}\quad \partial \Omega,
    \end{array}
    \right.
\end{equation}
where $R$ and $S$ are smooth functions.
\newline

We shall prove that this problem is solvable for $n=3$ in any sufficiently small convex domain,
and for general $n$ in a sufficiently small ball, in the sense that their volumes
and their slab diameter are sufficiently small. Recall that for a domain $\Omega\subset \mathbb{R}^n$, the {\it slab diameter}
is the smallest distance between two parallel hyperplanes such that
$\Omega$ is contained in the region in between them. Now state our first result.
\medskip

\begin{theorem}
\label{thm:main}
Let $\Omega\subset \mathbb{R}^n$, and $R,S:\Omega\longrightarrow \mathbb{R}^n$ be smooth 
functions such that $\left\|R\right\|_{\infty}\leq \Lambda$ and $\left\|S\right\|\leq \gamma$, for some $\Lambda, \gamma>0$.
Assume that:
\begin{enumerate}
    \item If $n=3$, suppose that $\Omega$ is a bounded convex domain with smooth boundary, and with 
    
    \begin{enumerate}
    \item[(a)] either
    \[\mbox{\rm Vol}\left(\Omega\right)\leq \min\left\{1,
\left(\frac{8}{4.76\pi^{\frac{2}{3}}\left(4\Lambda+\gamma\right)}\right)^3
\right\} \quad \mbox{and}\quad
\delta\leq \min\left\{1, 1.25\mbox{Vol}\left(\Omega\right)^{\frac{1}{3}},\frac{2}{0.75\Lambda+1}\right\}
,\]
\item[(b)] or 
\[
\mbox{\rm diam}\left(\Omega\right)\leq \min\left\{1, \frac{4}{4.76\pi^{\frac{2}{3}}\left(2.5\Lambda+\gamma\right)}, 
\frac{4}{1.5\Lambda+\sqrt{2}}\right\}.
\]
\end{enumerate}
    \item For $n\geq 4$, if $\Omega$ is a ball of radius $d$, with
    \[
    d\leq \min\left\{\frac{1}{2},\frac{n-1}{C_n\left(2.5\Lambda+\gamma\right)},
    \frac{n-1}{1.5\Lambda+\sqrt{2}}\right\}.
    \]
\end{enumerate}
Then the nonlinear Dirichlet problem (\ref{eq:Yamabe_nonestandard}) has a smooth solution.
\end{theorem}
\medskip

\noindent{\bf Remark} The estimate given for the volume of $\Omega$ and its slab diameter, and of its diameter,
are not sharp. It was given mostly to show that this estimate can be obtained explicitly without much effort from the proof
of the theorem. For instance, the condition $\delta\leq 1.25\mbox{Vol}\left(\Omega\right)^{\frac{1}{3}}$
holds whenever the largest ball that can be inscribed in $\Omega$ has radius $\delta/2$ and exists a diameter of this ball that hits the boundary of $\Omega$ at its two endpoints. This estimate
for the volume of $\Omega$ allows to apply the result to thin domains with large diameter, this in contrast to the results in \cite{RosenbergXu}.

\medskip
The equation above is equivalent to the usual form of the Yamabe equation (\ref{eqn: Yamabe_nongrad}), with boundary condition $u=1$. In fact, in this case we obtain:
\begin{equation}
\label{eq:Yamabe_standard_form}
     \left\{
    \begin{array}{l}
    \Delta u = Su-\dfrac{n-2}{2\left(n-1\right)}Ru^{\frac{n+2}{n-2}},
    \quad \mbox{in} \quad \Omega,\\
    u=1,\quad \mbox{on}\quad \partial \Omega.
    \end{array}
    \right.
\end{equation}

As a corollary from Theorem \ref{thm:main}, we obtain the following solvability result for 
a Yamabe-type problem
on a convex domain in $\mathbb{R}^3$ and small balls in $\mathbb{R}^n$.
\medskip

\begin{theorem}
\label{thm:main2}
Let $\Omega\subset \mathbb{R}^n$ and $S:\Omega\longrightarrow \mathbb{R}$
be a smooth function such that $\left\|S\right\|_{\infty}\leq \gamma$.
Assume that:
\begin{enumerate}
    \item If $n=3$, let $\Omega$ be a bounded convex domain of smooth boundary, with 
    
    \begin{enumerate}
        \item [(a)] either
    \[\mbox{\rm Vol}\left(\Omega\right)\leq 
\left(\frac{8}{4.76\pi^{\frac{2}{3}}\left(1+\gamma\right)}\right)^3 
\quad \mbox{and}\quad \delta\leq \min\left\{1,1.25\mbox{\rm Vol}\left(\Omega\right)^{\frac{1}{3}}\right\},
\]
 \item[(b)] or
 \[
\mbox{\rm diam}\left(\Omega\right)\leq \min\left\{1, \frac{4}{4.76\pi^{\frac{2}{3}}\left(0.625+\gamma\right)}\right\}.
\]

\end{enumerate}
\medskip
    \item For $n\geq 4$, assume $\Omega$ is a ball of radius $d$, with 
    \[
    d\leq \min\left\{\frac{1}{2},\frac{n-1}{C_n\left(0.625+\gamma\right)}
   % ,
    %\frac{\left(n-1\right)}{C_n\left(0.375+\sqrt{2}\right)}
    \right\}.
    \]
\end{enumerate}
Then for any $-\infty<c<\infty$ there exists a $\lambda \in \mathbb{R}$ such that 
\[
\left\{
\begin{array}{l}
\Delta f = -\dfrac{1}{2\left(n-1\right)}\left[\lambda e^{2f}+\left(n-1\right)
    \left(n-2\right)\left|\nabla f\right|^2\right]+S\left(x\right),
    \quad \mbox{in} \quad \Omega\\
    f=c,\quad \mbox{on}\quad \partial \Omega,
\end{array}
\right.
\]
has a smooth solution.
\end{theorem}
As a consequence of the previous result, we obtain that in small enough domain, for any constant $\varphi>0$,
and any $\lambda$ small enough,
that there is a $u>0$ which satisfies:
\begin{equation}
\label{eq:Yamabe_standard_form1}
     \left\{
    \begin{array}{l}
    \Delta u = Su-\dfrac{n-2}{2\left(n-1\right)}\lambda u^{\frac{n+2}{n-2}},
    \quad \mbox{in} \quad \Omega\\
    u=\varphi,\quad \mbox{on}\quad \partial \Omega.
    \end{array}
    \right.
\end{equation}
\medskip
	
Let us make a few comments on the previous result. First, notice that one of the solvability criteria
for the Yamabe problem given above depends on the volume and the slab diameter of the domain, and not on its diameter,
in contrast to the requirements in \cite{RosenbergXu}. Thus we can consider, in the
case of a convex domain in $\mathbb{R}^3$, as we said before, a very
long and thin domain. Also, it is important to notice
that our criterion for solvability is independent of $c$ (and hence of $\varphi$), which 
is not obvious at all, and that to get positivity
of $u=e^{2f}$, in the usual form of the Yamabe equation, is trivial in our context. 

\medskip
Finally, using our methods we give the following
geometric application (as it is, the following theorem can be obtained 
using the Inverse Function Theorem as well; however, we wanted to give 
a different proof).
\begin{theorem}
\label{thm:geometric_application}
Any bounded subdomain $\Omega\subset \mathbb{R}^n$
of smooth boundary admits a conformal deformation 
such that the resulting metric has constant scalar curvature, and the 
deformation can be chosen so that the sign of the scalar 
curvature is either positive or negative. Furthermore, the
deformation can be arranged so that the
metric of the boundary remains invariant, that is, $\partial M$ keeps
the metric it inherits from $\mathbb{R}^n$.
\end{theorem}

\medskip
Recently, by means of an iterative method and a maximum principle for Dirichlet problems we solved the Poisson equation with boundary conditions on bounded domains and some unbounded domains
which are bounded in one direction. The preprint can be found in \cite{CortissozTorres}. The present work was inspired by those results and the recent preprint \cite{RosenbergXu}, by S. Rosenberg and J. Xu. They considered the Yamabe equation on a Riemannian domain $(\Omega, g)$ of $\mathbb{R}^n$, where $g$ is a Riemannian metric that can be extended smootly to the boundary of $\Omega$. They proved that if the volume with respect to $g$ and the diameter of $\Omega$ are small enough, then the equation (\ref{eqn: Yamabe_nongrad}) admits at least one solution, with a constant positive boundary condition.

\medskip
Rosenberg an Xu also proposed an iterative method for solving the Yamabe equation. Nevertheless, their approach is different since they use the equation in its standard form (\ref{eqn: Yamabe_nongrad}). Furthermore, besides proving existence and regularity, the difficulty using this equation is to prove that $u$ is positive, so $u$ can be used as a conformal factor.

\medskip
If $R\geq 0$ and $S\leq 0$ in the equation above, the positivity of $u$ would be a consequence of the Maximum Principle. However, presented in the form (\ref{eq:Yamabe_nonestandard}) this issue actually disappears
without any regard to the signs of the functions $R$ and $S$. Indeed, we think that the arguments
in \cite{RosenbergXu} can be greatly simplified: First, by using the gradient form of the Yamabe equation; 
and second, by using a gradient estimate as the one given in the key Lemma \ref{lemma:gradient_green} below. In fact, as soon as there is a gradient estimate as 
the one in Lemma \ref{lemma:gradient_green}
or Lemma \ref{lemma:diameter}, our method can be used to prove existence of solutions to a Dirichlet problem for an elliptic operator with some well-behaved nonlinearities in the gradient of the unknown function. 
\newline

The reader 
will find the proof of Theorem \ref{thm:main} part (1), where the iteration technique is discussed,
in Section \ref{section:part1}; the proof of Theorem \ref{thm:main} part (2) is in Section
\ref{section:part2}; and the proofs of Theorems \ref{thm:main2} and \ref{thm:geometric_application}
are found in Section \ref{section:thm2_thm3}.

\section{Proof of Theorem \ref{thm:main} part (1)}
\label{section:part1}

The present work keep the following notations. The $L^2$-norm on a domain $\Omega\subseteq\mathbb{R}^n$ will be denoted by:
\[
\left\|f\right\|:=\left(\int_{\Omega}f^2 \right)^{1/2},
\]
integrating with respect to the Lebesgue measure of $\mathbb{R}^n$. Denote by $L^{\infty}(\Omega)$ the Banach space of bounded functions with the norm:
\[
\left\|f\right\|_{\infty}:=\sup_{\Omega}|f|.
\]

For the Sobolev spaces $W^{1, p}(\Omega)$ we use the norm:
\[
\left\|f\right\|_{1, p}=\left(\int_{\Omega}\sum_{|\beta |\leq k}|D^{\beta}f|^p\right)^{1/p}.
\]
with respect the Lebesgue measure of $\mathbb{R}^n$, for $p\geq 1$ and any multi-index $\beta=(\beta_1, \beta_2, \dots, \beta_s)$, where $|\beta|=\beta_1+ \beta_2+ \dots+ \beta_s$. As usual, for $C^1_0(\Omega)$, the space of $C^1$-class functions with compact support in $\Omega$, we will denote by $H^1_0(\Omega)$ its closure in $W^{1, 2}(\Omega)$. 
\newline

We have the following basic estimate.
\begin{lemma}
\label{lemma:gradient_green}
Let $\Omega\subset \mathbb{R}^3$ convex.
Let $h$ be a smooth function, and consider the solution to the
Dirichlet problem
\[
\Delta u=h \quad \mbox{in}\quad\Omega,\quad u=0 \quad \mbox{in}\quad \partial \Omega.
\]
There is a constant $C>0$ such that 
\[
\left\|\nabla u\right\|_{\infty}\leq CV^{\frac{1}{3}}\left\|h\right\|_{\infty},
\]
where $V$ is the volume of $\Omega$. The constant $C$ can be taken as $4.76\pi^{\frac{2}{3}}$.
\end{lemma}
\begin{proof}
Let $G$ be the Green's function of the Dirichlet Laplacian in $\Omega$. Then, by G. C. Evans (Theorem 2 of \cite{Evans}), we have that
\[
\int_{\Omega}\left|\nabla G\left(x,x'\right)\right|\,dx'\leq 
4.76\left(\pi^2 \mbox{Vol}\left(\Omega\right)\right)^{\frac{1}{3}}\leq 4.76\pi^{\frac{2}{3}}d.
\]
From this we obtain that 
\[
\left|\nabla u\left(x\right)\right|\leq \int_{\Omega}\left|\nabla G\left(x,x'\right)\right|
\left|h\left(x'\right)\right|\,dx',
\]
and hence
\[
\left\|\nabla u\right\|_{\infty}
\leq 
4.76\left(\pi^2 \mbox{Vol}\left(\Omega\right)\right)^{\frac{1}{3}}\left\|h\right\|_{\infty},
\]
which is what we wanted to prove.

\end{proof}

\subsection{Iteration procedure}
To produce a solution to the Yamabe equation, we define the following iteration:
\begin{equation}
\label{eqn: Iterprocess}
\left\{
\begin{array}{l}
\Delta f_{k+1}=-\dfrac{1}{2\left(n-1\right)}\left[Re^{2f_k}+\left|\nabla f_k\right|^2\right]
+S,
\quad \mbox{in}\quad \Omega,\\
f_{k}=0, \quad \mbox{on}\quad \partial \Omega,
\end{array}
\right.
\end{equation}
starting with $f_0=0$. For $R$ and $S$ smooth, each of these problems have a smooth solution.
\newline

Using Lemma (\ref{lemma:gradient_green}) we may now estimate as follows, using the conventions $\Lambda=\left\|R\right\|_{\infty}$,
 $\gamma = \left\|S\right\|_{\infty}$ and $A=\mbox{Vol}\left(\Omega\right)^{\frac{1}{3}}$, for each $k\geq 1$:
\begin{eqnarray*}
\left\|\nabla f_{k+1}\right\|_{\infty}&\leq& \frac{CA}{2\left(n-1\right)}
\left[\Lambda e^{2\left\|f_k\right\|_{\infty}}+\left\|\nabla f_{k}\right\|_{\infty}^2+\gamma\right]\\
&\leq& \frac{CA}{2\left(n-1\right)}
\left[\Lambda e^{\delta\left\|\nabla f_k\right\|_{\infty}}+\left\|\nabla f_{k}\right\|_{\infty}^2
+\gamma\right],
\end{eqnarray*}
where $\delta>0$ is the {\it slab diameter} of $\Omega$. We shall assume that
\[
\delta \leq 2\left(\frac{\mbox{Vol}\left(\Omega\right)}{\omega_3}\right)^{\frac{1}{3}}\leq 1.25 A,
\]
where $\omega_3$ is the volume of the unit ball in $\mathbb{R}^3$. Other
restrictions on $\delta>0$ are possible; we made this choice 
because it has a geometric meaning as explained in the introduction: it holds whenever the largest ball that can
be inscribed in $\Omega$ has $\delta/2$ as its radius and there is 
a diameter of this ball that hits the boundary of $\Omega$ at its two endpoints. 

\medskip
From this we obtain that if $A$ is small enough, then there exists a $K>0$ such that for the
function
\[
f\left(t\right)=\frac{CA}{2\left(n-1\right)}\left(\Lambda e^{1.25At}+t^2+\gamma\right)
\]
satisfies that $f\left(t\right)\leq K:=K\left(A, \Lambda,\gamma\right)$ whenever $0\leq t\leq K$. The proof of this
statement is elementary. Take $K$ as the smallest solution to the
equation $t=f\left(t\right)$, which 
does exist if $A$ is small enough, and as $f$ is increasing the claim follows. What
we can also show, and will be needed in what follows, is that the smallest root
of $t=f\left(t\right)$ does not increase without
bound as $A\rightarrow 0$. In fact $K$ is $O\left(A\right)$ for $\Lambda$
and $\gamma$ fixed. Indeed, we can write
\begin{equation}
\label{eq:fixedpoint_eq}
f\left(t\right)-t=\dfrac{CA}{2\left(n-1\right)}t^2+
\left(\frac{1.25CA^2\Lambda}{2\left(n-1\right)}-1\right)t+\frac{CA}{2\left(n-1\right)}
\left(\gamma+\Lambda+O\left(A^2\right)\right),
\end{equation}
and our assertion is a consequence of the quadratic formula. We shall give a more
explicit estimate at the end of this section.
\newline

To proceed, we next consider
\[
\Delta \left(f_{k+1}-f_k\right)=
-\dfrac{1}{2\left(n-1\right)}\left[R\left(e^{2f_k}-e^{2f_{k-1}}\right)+
\left|\nabla f_k\right|^2-\left|\nabla f_{k-1}\right|^2\right],
\]
with the boundary condition $f_{k+1}-f_k=0$. 
\medskip

\noindent On the other hand, using the mean value theorem, we can estimate
\begin{eqnarray*}
\left|e^{2f_k}-e^{2f_{k-1}}\right|&\leq&e^{2\max\left\{\left\|f_k\right\|_{\infty},
\left\|f_{k-1}\right\|_{\infty}\right\}}\left|f_k-f_{k-1}\right|\\
&\leq&e^{\delta\max\left\{\left\|\nabla f_k\right\|_{\infty},
\left\|\nabla f_{k-1}\right\|_{\infty}\right\}}\left|f_k-f_{k-1}\right|\leq 
e^{\delta K}\left|f_k-f_{k-1}\right|,
\end{eqnarray*}
where we have used that for all $k$, $\left\|\nabla f_k\right\|_{\infty}\leq K$,
which can be proved by induction starting from the fact that $f_0=0$. Next, we can estimate
\begin{eqnarray*}
\left|\nabla f_k\right|^2-\left|\nabla f_{k-1}\right|^2&=& 
\left|\nabla f_k+\nabla f_{k-1}\right|\left|\nabla f_k-\nabla f_{k-1}\right|\\
&\leq& 2K\left|\nabla f_k-\nabla f_{k-1}\right|.
\end{eqnarray*}

Then, multiplying the expresion for $\Delta \left(f_{k+1}-f_k\right)$ by $f_{k+1}-f_k$, integrating,
and using the Cauchy-Schwartz inequality yields
\[
\left\|\nabla f_{k+1}-\nabla f_k\right\|^2\leq 
\frac{1}{2\left(n-1\right)}\left(\Lambda e^{K\delta}\left\|f_{k+1}-f_{k}\right\|
\left\|f_{k}-f_{k-1}\right\|
+2K\left\|f_{k+1}-f_{k}\right\|\left\|\nabla f_{k}-\nabla f_{k-1}\right\|\right).
\]
Using the Poincar\'e inequality, 
\[
\left\|f_j-f_{j}\right\|\leq \frac{\delta}{\sqrt{2}}\left\|\nabla f_j-\nabla f_{j-1}\right\|,
\]
where $\delta$, as before, stands for the slab diameter of the domain, we get
\begin{eqnarray*}
\left\|\nabla f_{k+1}-\nabla f_k\right\|^2&\leq& 
\frac{1}{2\left(n-1\right)}\left(\frac{\delta^2\Lambda e^{K\delta}}{2}\left\|\nabla f_{k+1}-\nabla f_{k}\right\|
\left\|\nabla f_{k}-\nabla f_{k-1}\right\|\right.\\
&&\left.
\quad 
\quad 
\quad
\qquad
\qquad
+\frac{2K\delta}{\sqrt{2}}\left\|\nabla f_{k+1}-\nabla f_{k}\right\|\left\|\nabla f_{k}-\nabla f_{k-1}\right\|\right),
\end{eqnarray*}
which gives after simplification
\begin{equation}
\label{ineq:contraction}
\left\|\nabla f_{k+1}-\nabla f_k\right\|\leq \frac{1}{2\left(n-1\right)}
\left(\frac{\delta^2\Lambda e^{K\delta}}{2}+\sqrt{2}\delta K\right)\left\|\nabla f_{k}-\nabla f_{k-1}\right\|.
\end{equation}
\medskip

Thus we have obtained the following:

\begin{prop}
Consider the sequence $\{f_k\}$ produced by the iteration process (\ref{eqn: Iterprocess}). Let $A={\rm Vol}(\Omega)^{\frac{1}{3}}$ and $\delta$ the slab diameter of $\Omega$. If $A$ and $\delta$ are small enough the following holds:
\begin{enumerate}
    \item At each step $f_k$ is smooth.
    \item $f_k$ is uniformly in $W^{1,\infty}\left(\Omega\right)$.
    \item The sequence $f_k$ converges in $H_0^{1}\left(\Omega\right)$.
\end{enumerate}
\end{prop}
\medskip

Hence, there is a uniform limit $f$, and actually, this $f_k\rightarrow f$ in every Hölder space
$C^{\alpha}(\Omega)$, for any $0<\alpha<1$ and the limit is actually Lipschitz. To prove that $f$ is smooth, define $u$ by the relation $u^{\frac{4}{n-2}}=e^{f}$. Then we obtain equation (\ref{eq:Yamabe_standard_form}), and as $u$ is Lipschitz, as $f$ is, by a standard bootstrapping argument we obtain that $u$ is smooth, and so is $f$. 
%We can avoid the use
%of the standard form of the Yamabe equation as follows: The left handside of () is in $L^2\left(\Omega\right)$
%and thus the righthandside is in $W^{2,2}$. Using the fact that $\left|\nabla f\right|$

\subsection{Estimating $\mbox{Vol}\left(\Omega\right)$ and $\mbox{diam}\left(\Omega\right)$}
The $O\left(A^2\right)$ term in (\ref{eq:fixedpoint_eq}) is given by
\[
\frac{1.25^2\Lambda A^2  e^{1.25A\xi} t^2}{2},\quad  0<\xi<t.
\]
We shall assume that $A\leq 1$, and thus if $A$ is small enough, as the roots of the 
resulting quadratic equation in (\ref{eq:fixedpoint_eq}) when we replace
the $O\left(A\right)$ by its largest
possible value, $0.8\Lambda A^2  e^{1.25A\xi}$,
are less than
\begin{equation}
\label{eq:fixed_point1}
\frac{CA}{8}\left(\Lambda+\gamma+3\Lambda\right),
\end{equation}
where we are also assuming that $t\leq 1$, then the smallest root 
of $f\left(t\right)-t$ will be bounded above also by (\ref{eq:fixed_point1}).
We make the assumption $t\leq 1$ self-consistent by arranging 
the previous expression to be smaller than 1. From this we get 
that by making 
\[
A\leq \min\left\{1,\frac{8}{C\left(4\Lambda+\gamma\right)}\right\},
\]
then we guarantee that the smallest fixed point of $f$ is less than the
expression given by (\ref{eq:fixed_point1}), which in turn is less than 1.
Since $\delta\leq 1.25A$, we have that in order to have the 
constant in the righthand side of (\ref{ineq:contraction}), we require that 
\[
\frac{1}{4}
\left(\frac{\delta^2\Lambda e^{\delta}}{2}+\sqrt{2}\delta \right)< 1,
\]
where we have used that $K\leq 1$. Also, as we shall require that $\delta\leq 1$, the previous inequality
will hold if
\[
\frac{1}{4}
\left(1.5 \delta\Lambda+\sqrt{2}\delta \right)\leq 1,
\]
which in turn will hold if we require that
\[
\frac{1}{4}\left(1.5\delta \Lambda +2\delta\right)\leq 1.
\]
From this we obtain that as long as
\[
\delta\leq \min\left\{1, \frac{2}{0.75\Lambda+1}\right\},
\]
the constant in the right-hand side of (\ref{ineq:contraction}) will be less than 1. 

\medskip
Putting all this together yields that if 
\[
\mbox{Vol}\left(\Omega\right)\leq \min\left\{1,
\left(\frac{8}{4.76\pi^{\frac{2}{3}}\left(4\Lambda+\gamma\right)}\right)^3\right\}
\quad \mbox{and}\quad
\delta\leq \min\left\{1, 1.25 \mbox{Vol}\left(\Omega\right)^{\frac{1}{3}}, \frac{2}{0.75\Lambda+1}\right\},
\]
then the Dirichlet problem (\ref{eq:Yamabe_nonestandard}) has a solution. 

\medskip
If we take into account that $A\leq\mbox{diam}\left(\Omega\right):=d$, we could have worked out
our estimates in terms of $d$. We indicate how this can be done. Our first estimate
for $\left\|\nabla f_{k+1}\right\|_{\infty}$ can be replaced by
\[
\left\|\nabla f_{k+1}\right\|_{\infty}\leq \frac{Cd}{2\left(n-1\right)}
\left[\Lambda e^{d\left\|\nabla f_k\right\|_{\infty}}+\left\|\nabla f_{k}\right\|_{\infty}^2
+\gamma\right],
\]
and thus we estimate the smallest fixed point of the function
\[
f\left(t\right)=\frac{Cd}{4}\left(e^{dt}+t^2+\gamma\right).
\]
Hence, reasoning as before, using $d$ instead of $A$, we find that we have
that
\[
d\leq \min\left\{1, \frac{4}{C\left(2.5\Lambda+\gamma\right)}\right\}
\]
guarantees that the smallest fixed point of $f$ is smaller than 1. Estimate (\ref{ineq:contraction}) in terms
of $d$ reads as
\[
\left\|\nabla f_{k+1}-\nabla f_k\right\|\leq \frac{1}{4}
\left(\frac{d^2\Lambda e^{Kd}}{2}+\sqrt{2}d K\right)\left\|\nabla f_{k}-\nabla f_{k-1}\right\|,
\]
where $K$ is the smallest fixed point of $f$. Thus, if we require $d\leq 1$ and the constant
in the right-hand side of the previous estimate to be $<1$, it suffices to have
\[
\frac{1}{4}\left(1.5 d\Lambda +\sqrt{2}d\right)\leq 1,
\]
that is, if
\[
d\leq \min\left\{1, \frac{4}{C\left(2.5\Lambda+\gamma\right)}, \frac{4}{1.5\Lambda+\sqrt{2}}\right\},
\]
then the Dirichlet problem (\ref{eq:Yamabe_nonestandard}) has a solution.

\section{Proof of Theorem \ref{thm:main} part (2)}
\label{section:part2}

For the case of balls of radius $d$ in $\mathbb{R}^n$, we can prove that if $G_r$ is the Green's function
for the Dirichlet Laplacian for the ball $B_r^n$
of radius $r$ centered at the origin then, it is related to the Dirichlet Laplacian of the unit
ball by the formula
\[
G_r\left(x,x'\right)=\frac{1}{r^{n-2}}G_{1}\left(\frac{x}{r},\frac{x'}{r'}\right).
\]
From this scaling property, the following estimate follows
\[
\int_{B^n_r}\left|\nabla G_r\left(x,x'\right)\right|\,dx'
\leq r\int_{B^n_1}\left|\nabla G_1\left(x,x'\right)\right|\,dx'\leq Cr,
\]
and thus we have the analogue of Lemma \ref{lemma:gradient_green}. 
\begin{lemma}
\label{lemma:diameter}
Let $h$ be a smooth function, and consider the solution to the
Dirichlet problem
\[
\Delta u=h \quad \mbox{in}\quad B^n_{d},\quad u=0 \quad \mbox{in}\quad 
\partial B_{d}.
\]
There is a constant $C_n>0$, which depends only on the dimension $n$, such that 
\[
\left\|\nabla u\right\|_{\infty}\leq C_n d\left\|h\right\|_{\infty}.
\]
The constant $C_n$ can be taken as
\[
\sup_{x\in B^n_1}\int_{B^n_1}\left|\nabla G_1\left(x,x'\right)\right|\,dx'.
\]
\end{lemma}
The rest of the proof is the same as the one given in the case of a bounded convex subdomain
of $\mathbb{R}^3$, and estimating the size of $d$ follows along the same lines as in estimating
the volume in the case of convex domains in $\mathbb{R}^3$, but working with $d$ instead of
working with $A$ and $\delta$ (and recall that $\delta=2d$).

\section{Proof of Theorems \ref{thm:main2} and \ref{thm:geometric_application}}
\label{section:thm2_thm3}

As mentioned before, Theorem \ref{thm:main2} is a corollary of Theorem \ref{thm:main}. Define $u=f-c$, for any $c\in (-\infty, \infty)$, and consider the 
Dirichlet problem
\begin{equation}
\label{eq:Yamabe_nonestandard2}
    \left\{
    \begin{array}{l}
    \Delta u = -\dfrac{1}{2\left(n-1\right)}\left[\lambda e^{2c}e^{2u}+\left(n-1\right)
    \left(n-2\right)\left|\nabla u\right|^2\right]+S\left(x\right),
    \quad \mbox{in} \quad \Omega\\
    u=0,\quad \mbox{on}\quad \partial \Omega,
    \end{array}
    \right.
\end{equation}
where we may rescale $\lambda$ in a way that $\left|\lambda\right|=\dfrac{1}{4}e^{-2c}$. Therefore Theorem \ref{thm:main2} follows from Theorem \ref{thm:main}, by taking $\Lambda=1/4$.
\newline

It remains to prove Theorem \ref{thm:geometric_application}. Let $\Omega$ be any bounded subdomain
of $\mathbb{R}^n$, and we denote its Dirichlet Green's function by $G_1\left(x,x'\right)$.
Assume that $0\in \Omega$, and define $\Omega_d$ by
\[
\Omega_d=d\cdot\Omega:=\left\{d\cdot x\in \mathbb{R}^n:\, x\in \Omega\right\}.
\]
Then the Dirichlet
Green's function of $\Omega_d$ is given by
\[
G_d\left(x,x'\right)=\frac{1}{d^{n-2}}G_1\left(\frac{x}{d},\frac{x'}{d}\right).
\]
Therefore, reasoning as above, the Dirichlet problem
\begin{equation}
\label{eq:Yamabe_nonestandard3}
    \left\{
    \begin{array}{l}
    \Delta u = -\dfrac{1}{2\left(n-1\right)}\left[\lambda e^{2u}+\left(n-1\right)
    \left(n-2\right)\left|\nabla u\right|^2\right],
    \quad \mbox{in} \quad \Omega_d\\
    u=0,\quad \mbox{on}\quad \partial \Omega_d,
    \end{array}
    \right.
\end{equation}
has a solution for $d$ and $\left|\lambda\right|$ small enough.
This means, that there is a conformal deformation of the Euclidean metric 
such that $\Omega_d$ has constant scalar curvature $\lambda$, and such that $\partial \Omega_d$
keeps its metric invariant. But then, by reescaling, i.e., considering 
$\Omega=\dfrac{1}{d}\cdot \Omega_d$,
we obtain that $\Omega$ admits a metric of constant scalar curvature $\lambda/d^2$ whereas 
the metric in $\partial \Omega$ remains as the metric it inherits from $\mathbb{R}^n$.

\end{document}